\documentclass[10pt]{amsart}

\calclayout

\usepackage{amsmath, amssymb}
\usepackage{amsfonts}
\usepackage{mathrsfs}
\usepackage[color, arrow,matrix,curve,cmtip,ps]{xy}
\usepackage{paralist}
\usepackage{quiver}
\usepackage{amsthm}
\usepackage{caption}
\usetikzlibrary{arrows,calc,matrix}
\tikzset{
curvarr/.style={
  to path={ -- ([xshift=2ex]\tikztostart.east)
    |- (#1) [near end]\tikztonodes
    -| ([xshift=-2ex]\tikztotarget.west)
    -- (\tikztotarget)}
  }
}
\tikzset{%
    symbol/.style={%
        draw=none,
        every to/.append style={%
            edge node={node [sloped, allow upside down, auto=false]{$#1$}}}
    }
}

\usepackage[final]{pdfpages}
\usepackage{dsfont,txfonts}
\usepackage{rotating}
\PassOptionsToPackage{hyphens}{url}\usepackage{hyperref}
\usepackage{enumitem}
\usepackage{graphicx}
\usepackage{mathtools}
\usetikzlibrary{arrows,chains,matrix,positioning,scopes}

\usepackage[T1]{fontenc}
\usepackage[variablett]{lmodern}
\usepackage{xcolor}
\usepackage{lipsum}
\usepackage[most]{tcolorbox}

\newtheorem{theorem}{Theorem}[section]
\newtheorem{lemma}[theorem]{Lemma}

\newtheorem{corollary}[theorem]{Corollary}
\theoremstyle{definition}

\newtheorem{remark}[theorem]{Remark}

\newtheorem{question}[theorem]{Question}

\everymath{\displaystyle}
\allowdisplaybreaks

\newcommand{\Map}{\mathrm{Map}}

\newcommand{\ab}{\mathrm{ab}}

\newcommand{\trc}{\mathrm{trc}}

\newcommand{\Z}{\mathbb Z}
\newcommand{\Res}{\operatorname{Res}}
\newcommand{\tr}{\operatorname{tr}}

\begin{document}
\title[The Becker--Gottlieb transfer is not functorial]{The Becker--Gottlieb transfer is not functorial: \\ An example based on units of group rings}
\author{Georg Lehner}
\author{Cary Malkiewich}
\author{Florian Naef}
\begin{abstract}
We show by explicit counterexample that the Becker--Gottlieb transfer cannot be made functorial.
A different counterexample using rational homotopy theory was discovered concurrently by the third author and Robin Stoll.
\end{abstract}
\maketitle

\setcounter{section}{1}

A \emph{perfect fibration} is a fibration of topological spaces $p\colon X \rightarrow Y$ whose fibers are finitely dominated, i.e.~a retract up to homotopy of a finite CW complex. For each perfect fibration, there is a \emph{Becker-Gottlieb transfer} $p^!$, a wrong-way map in the stable homotopy category of the form
\[ \begin{tikzcd} p^!\colon \Sigma^\infty_+ Y \rar &  \Sigma^\infty_+ X. \end{tikzcd} \]
Informally, it is defined by summation along the fibers. More precisely, it is defined as the pushforward along $Y \to *$ of the trace of the diagonal map $X \to X \times_Y X$ in the symmetric monoidal category of parametrized spectra over $Y$. See e.g.~ \cite{becker1976transfer}, \cite[15.2.1]{ms}, \cite[7.1.5]{malkiewich_parametrized}, \cite[\S 4]{KleinMalkiewich2018}, \cite[Ex 4.9]{lm}, or \cite[\S 6]{CarmeliCnossenRamziYanovski2025} for more details. In the special case that $p$ is a finite-sheeted covering space, this transfer can be described in many other equivalent ways, and appears all over stable homotopy theory.

In this paper we address the open question of whether the Becker--Gottlieb transfer is functorial:
\begin{question}\label{main_question}
    Given a pair of composable perfect fibrations
    \[ \begin{tikzcd}
        X \rar{p} & Y \rar{q} & Z,
    \end{tikzcd} \]
do we have an equality $(q \circ p)^! = p^! \circ q^!$ as morphisms in the stable homotopy category?
\end{question}

For a long time it seemed that Question \ref{main_question} should have a positive answer. It is easily seen to be true for finite-sheeted covers, a fact that has been used extensively in calculations. It is also known to be true for smooth fiber bundles with compact manifold fibers, or more generally fiberwise compact ENRs, by a theorem of Dold \cite[Thm 8.7]{dold_fixed_point_transfer}. Forthcoming work of Ramzi--Volpe--Wolf generalizes this even further using the new technology of dualizable stable $\infty$-categories. However, none of these approaches have led to a positive answer in the most general case of perfect fibrations.

Question \ref{main_question} attracted considerable interest in the previous decade, culminating in a claimed proof of functoriality for all perfect fibrations by the second author and John Klein \cite{KleinMalkiewich2018}. Subsequently, Carmeli and Cnossen found a mistake in this proof \cite{KleinMalkiewich2022Corrigendum}, and indeed the proof strategy cannot be correct for formal reasons, as explained in \cite[\S 6]{CarmeliCnossenRamziYanovski2025}.
Later work of the second author with Klein and Ramzi shows that the Becker-Gottlieb transfer is functorial on the level of $\pi_0$ \cite{KleinMalkiewichRamzi2023}, but the question of functoriality in general has remained open.

In this paper we resolve Question \ref{main_question} in the negative: the transfer cannot be made functorial along arbitrary pairs of perfect fibrations.
\begin{theorem}\label{main_thm}
    There exists a pair of perfect fibrations $p\colon X \rightarrow Y$, $q\colon Y \rightarrow Z$ such that $( q \circ p)^! \neq p^! \circ q^!$ as maps in the stable homotopy category.
\end{theorem}

In fact, we give an explicit counterexample in which $p\colon X \rightarrow Y$ is a finite-sheeted covering space and $q\colon Y \rightarrow S^1$ is a fibration with finite CW fiber $V$.
The monodromy map $V \to V$ is a homotopy equivalence but not a simple homotopy equivalence, and therefore $q$ is not weakly equivalent to a \emph{fiber bundle} with finite CW fiber. This is what prevents the usual geometric argument in e.g.~\cite[Thm 8.7]{dold_fixed_point_transfer} from proving that $(q \circ p)^! = p^! \circ q^!$.

At the same time that this investigation was carried out, the third author and Robin Stoll constructed a different counterexample in which $p$ is an $S^1$-bundle and $q$ is a fibration with base $S^{n+1}$ and with finite CW fiber \cite{naef_stoll_counterexample}. Since our counterexamples are based on different frameworks, we have written them up as separate papers, but we emphasize that the discovery of these counterexamples happened at the same time. The presence of two unrelated counterexamples now leaves little doubt that the answer to Question \ref{main_question} is negative.

\begin{remark}\label{pullback_functoriality}
Although we disprove Question \ref{main_question} in general, it is true in many if not most cases of interest, and there is still much work to do in giving a complete accounting of when functoriality holds. For instance, \cite{KleinMalkiewich2022Corrigendum} shows that there can be no counterexample in which $q$ is a finite-sheeted covering space, but our counterexample shows that $p$ may be a finite-sheeted covering space. Forthcoming work of Ramzi--Volpe--Wolf also proves functoriality for proper, locally contractible maps between locally contractible and locally compact Hausdorff spaces, generalizing \cite[Thm 8.7]{dold_fixed_point_transfer}.

We also point out that the Becker--Gottlieb transfer is well-known to be a functor in a different sense: for each homotopy pullback square of spaces in which the vertical maps are perfect fibrations
\[ \begin{tikzcd}
    W \dar{q} \rar{r} & X \dar{p} \\
    V \rar{s} & Y,
\end{tikzcd} \]
we have an equality $(\Sigma^\infty_+ r) \circ (q^!) = (p^!) \circ (\Sigma^\infty_+ s)$ as maps in the stable homotopy category. This is straightforward to prove from the formal description of the transfer in e.g.~\cite[Ex 4.9]{lm} or \cite[\S 6]{CarmeliCnossenRamziYanovski2025}.
\end{remark}

\subsection*{AI declaration}

The authors developed the strategy for the counterexample, and the computational lemmas involved, without the use of AI in 2023--2024. The authors then used a large language model to find the counterexample itself.
The model was also used to translate this counterexample into the pair of maps for which the Becker--Gottlieb transfer does not compose. However, we stress that the text presented here has been entirely written by the three authors, and it reflects their human understanding of the investigation. All computations have been independently verified, and in some cases simplified from the LLM version.

\subsection*{Acknowledgements}
We would like to thank Robin Stoll for agreeing to release \cite{naef_stoll_counterexample} at the same time as the present paper.

The first author was funded by the Deutsche Forschungsgemeinschaft (DFG, German Research Foundation) – Project-ID 427320536–SFB 1442, as well as under Germany's Excellence Strategy EXC 2044/2 - 390685587, Mathematics Münster: Dynamics-Geometry-Structure. The second author was partially supported by the National Science Foundation (DMS-2506430) and Simons Foundation International [SFI-MPS-TSM-00013425, CM].

\section{Proof of the main theorem}

Our strategy for the disproof of Question \ref{main_question} is to use the following result of Raptis.

\begin{theorem}[{\cite[Theorem 4.2]{Raptis2021TransferMaps}}]
Assume that for any two perfect fibrations $p\colon X \rightarrow Y$, $q\colon Y \rightarrow Z$, the equality $( q \circ p)^! = p^! \circ q^!$ always holds as morphisms in the stable homotopy category. Then for any perfect fibration $p\colon X \rightarrow Y$, the square of homotopy groups
\begin{equation}\label{raptis_square}
\begin{tikzcd}
	{\pi_n(A(Y))} & {\pi_n(A(X))} \\
	{\pi_n(\Sigma^\infty_+(Y))} & {\pi_n(\Sigma^\infty_+(X))}
	\arrow["{\mathrm{tr}_p}", from=1-1, to=1-2]
	\arrow["{\trc_Y}"', from=1-1, to=2-1]
	\arrow["{\trc_X}", from=1-2, to=2-2]
	\arrow["{p^!}", from=2-1, to=2-2]
\end{tikzcd}
\end{equation}
commutes for every $n \geq 0$. Here
\begin{itemize}
    \item $\mathrm{tr}_p\colon A(Y) \rightarrow A(X)$ is the transfer on Waldhausen's algebraic $K$-theory of spaces \cite{1126} which sends each retractive space $Z$ over $Y$ to the pullback $p^*Z = X \times_Y Z$, and
    \item $\trc_X\colon A(X) \rightarrow \Sigma^\infty_+ X $ is the projection from $A$-theory back to stable homotopy.
\end{itemize}
\end{theorem}

More precisely, we construct $\trc_X$ by taking the Dennis trace map from $A(X)$ to topological Hochschild homology $\mathrm{THH}(X) \simeq \Sigma^\infty_+ LX$, where $LX = \Map(S^1,X)$ is the free loop space of $X$, and then we compose with the map $LX \to X$ that evaluates a loop at the basepoint of $S^1$. 
Therefore, to show that the Becker--Gottlieb transfer is not functorial, it suffices to find an element of $\pi_1(A(Y))$ on which the above square fails to commute. To set this up, recall that for any based connected space $X$ we have an isomorphism $\pi_1(A(X)) \cong K_1(\mathbb{Z}[\pi_1(X)])$, which is the group of all invertible $\mathbb{Z}[\pi_1(X)]$-matrices up to elementary row operations. We also have a splitting $\pi_1( \Sigma^\infty_+ X) \cong \{ \pm 1\} \oplus \pi_1(X)^\ab$, defined using the basepoint of $X$ to split off the copy of $\pi_1(\Sigma^\infty_+ *) \cong \{\pm 1\}$.

The following three lemmas allow us to compute the maps in the square \eqref{raptis_square} for any finite-sheeted covering space $X \to Y$ at the level of $\pi_1$. The first is a direct consequence of the general description of the transfer in $A$-theory, see for instance \cite[Section 3]{Raptis2021TransferMaps}. 
\begin{lemma}\label{lemma1}
    Suppose $p\colon X \to Y$ is a finite-sheeted cover of connected CW complexes with fundamental groups $H = \pi_1(X)$, $G = \pi_1(Y)$, and $[G:H] = n$. Then on $\pi_1$, the $K$-theory transfer $\tr_p$ is the restriction map
    \[\begin{tikzcd}
        K_1(\mathbb{Z}[\pi_1(Y)]) \rar{\Res} & K_1(\mathbb{Z}[\pi_1(X)])
    \end{tikzcd} \]
    that regards each invertible $(k \times k)$ matrix over $\mathbb{Z}[\pi_1(Y)]$ as an automorphism of a free right $\mathbb{Z}[\pi_1(Y)]$-module, then forgets to the right $\mathbb{Z}[\pi_1(X)]$-action, giving an invertible $(nk \times nk)$ matrix.
\end{lemma}

Next we describe the projection from $A$-theory to stable homotopy. Define
\[ \epsilon\colon \Z[G] \to \Z, \qquad \log\colon \Z[G] \to G^\ab \]
as follows. The map called $\epsilon$ is the augmentation, induced by the homomorphism $G \to 1$. The map called $\log$ is the unique additive homomorphism that sends each element $g \in G$ to itself. We call this a ``logarithm'' because we think of elements of $G^\ab$ additively, so that $\log$ sends multiplication of elements of $G$ to addition of elements of $G^\ab$.
\begin{lemma}\label{lemma2}
    If $G = \pi_1(Y)$, then the projection map $\trc_Y\colon A(Y) \rightarrow \Sigma^\infty_+ Y$ is calculated on $\pi_1$ by sending each invertible matrix $M$ with entries in $\mathbb{Z}[G]$ to the element
    \[ \Big(\det(\epsilon(M)) \ , \ \log\big(\tr(M\epsilon(M)^{-1})\big)\Big) \in \{\pm 1\} \oplus G^\ab. \]
    In particular, when $\epsilon(M)$ is an identity matrix, this gives
    \[ \Big( +1 \ , \ \log\big(\tr(M)\big)\Big) \in \{\pm 1\} \oplus G^\ab. \]
\end{lemma}

\begin{proof}
    The first coordinate can be computed from the naturality of $\trc$ by applying the map $Y \to *$, which sends $M \mapsto \epsilon(M)$ and then identifies $K_1(\Z) \cong \{\pm 1\}$. The formula for the second coordinate is extracted from the classical definition of the Dennis trace from e.g.~\cite{madsen_traces}, composed with the augmentation into $B^{\otimes} \Z[G]$ that applies $\epsilon$ to the first tensor factor:
    \[ \begin{tikzcd}[row sep = {1.5em}]
        B^\times GL_n(\Z[G]) \dar & (M_1,\ldots,M_k) \dar[|->] \\
        B^{\mathrm{cyc},\times} GL_n(\Z[G]) \dar & (M_0;M_1,\ldots,M_k) \dar[|->] & M_0 = M_k^{-1}\cdots M_1^{-1} \\
        B^{\mathrm{cyc},\otimes} M_n(\Z[G]) \dar & M_0 \otimes M_1 \otimes \cdots \otimes M_k \dar[|->] \\
        B^{\mathrm{cyc},\otimes} \Z[G] \dar & \sum_{1 \leq i_0,\ldots,i_k \leq n} M_0^{i_ki_0} \otimes M_1^{i_0i_1} \otimes \cdots \otimes M_k^{i_{k-1}i_k} \dar[|->] \\
        B^{\otimes} \Z[G] & \sum_{1 \leq i_0,\ldots,i_k \leq n} \epsilon(M_0)^{i_ki_0} \otimes M_1^{i_0i_1} \otimes \cdots \otimes M_k^{i_{k-1}i_k}.
    \end{tikzcd} \]
    Here $B$ refers to the bar construction and $B^{\mathrm{cyc}}$ refers to the cyclic bar construction. On first homology, this takes the class given by $M$ to the class given by the trace of the matrix $\epsilon(M^{-1})M$, which is also the trace of $M\epsilon(M)^{-1}$. (This formula for the second coordinate is also consistent with \cite[Proposition 1.2 and Definition 1.3]{GeogheganNicas1999}, which takes $M \mapsto M^{-1} \otimes M$ as the definition of the Dennis trace on $K_1$.)
\end{proof}
The final lemma describes the Becker--Gottlieb transfer itself on $\pi_1$. Actually, we don't need this lemma to verify our counterexample to the commutativity of \eqref{raptis_square}, because we only apply this homomorphism to the identity element in that case. But we will use it later when we construct an explicit pair of perfect fibrations that realizes the counterexample. It is also useful for computing $\pi_1$ of the diagram \eqref{raptis_square} in other examples.
\begin{lemma}\label{lemma3}
    For a finite-sheeted cover $X \to Y$ with fundamental groups $H = \pi_1(X)$ and $G = \pi_1(Y)$ as above, the Becker--Gottlieb transfer on $\pi_1$ is the map of abelian groups
    \[ \begin{tikzcd} \{ \pm 1 \} \oplus G^\ab \rar & \{ \pm 1 \} \oplus H^\ab \end{tikzcd} \]
    that is the sum of the following three homomorphisms:
    \begin{itemize}
        \item the homomorphism $\{ \pm 1 \} \to \{ \pm 1 \}$ raises to the power of the index $[G:H]$,
        \item the classical group theory transfer $G^\ab \to H^\ab$ originally due to Schur \cite{schur}\footnote{Explicitly, this is given by $g \mapsto \textstyle\prod_{g_i H \in G/H} g_j^{-1}gg_i$, where the $g_i$ are a fixed set of coset representatives, and for each $i$, $g_j$ is the coset representative of the coset $gg_iH \in G/H$.}, and 
        \item a cross-term homomorphism $G^\ab \to \{ \pm 1 \}$ sending $g$ to $(-1)^{|G/H| - |\left\langle g \right\rangle \backslash G/H|}$.
    \end{itemize}
\end{lemma}

\begin{proof}
    Every finite-sheeted cover is a pullback of one of the form $EG/H \to EG/G$ with $EG$ a free contractible $G$-CW complex, so using Remark \ref{pullback_functoriality}, we may without loss of generality assume our covering space is $EG/H \to EG/G$. The first homomorphism is easily computed by the functoriality of the transfer recalled in Remark \ref{pullback_functoriality}, and by the same functoriality the homomorphism $\{\pm 1\} \to H^\ab$ must vanish. The second homomorphism arises by lifting the loop $g$ in $BG$ to $[G:H]$ different paths in $EG/H$, the $i$th one going from $g_i H$ to $g_j H$. The homology class of this chain is then the sum of the loops we get by connecting each path back to the basepoint $1H$, hence the terms $g_j^{-1}gg_i$. The third homomorphism is computed by pulling back the covering space to a single circle $S^1 \subseteq BG$, giving a covering space $P \to S^1$ with fiber $G/H$ and monodromy given by left multiplication by $g$. For each orbit of size $n$, this contributes $(n-1)$ copies of the Hopf map $\eta$, by the classical computation of the transfer for the $n$-fold covering map $S^1 \to S^1$. Adding this up over all orbits (of which there are $|\left\langle g \right\rangle \backslash G/H|$), we get $(|G/H| - | \left\langle g \right\rangle \backslash G/H|)$ copies of the Hopf map.
\end{proof}

The above lemmas were proven earlier by the authors in unpublished work as they searched for counterexamples; the input from an LLM was to locate the following specific counterexample. Let
\[
\begin{aligned}
G&=\langle a,s \mid a^3=s^{16}=1,\ sas^{-1}=a^{-1}\rangle,\\
H&=\langle s \mid s^{16} = 1 \rangle\cong  C_{16}.
\end{aligned}
\]
We think of $H$ as a subgroup of $G$ with index $[G:H]=3$, and we identify both of their abelianizations with $\Z/16$. The inclusion $H \leq G$ induces a map $p\colon  BH \to BG$, which is a 3-sheeted cover. Let
\[
\begin{array}{rclcl}
    \beta &=& (s^8-1)(s + s^{-1}) = -s + s^7 + s^9 - s^{15}, \\
    \alpha &=& 1 + (1 - s^8)(a - 1)(1 - s - s^{-1}) \\
    &=& s^8 - \beta + a(1 - s^8 + \beta) \\
    &=& s^8 - \beta + (1 - s^8)a + \beta a^2 & \textup{in} & \Z[G].
\end{array}
\]
Here $\beta$ has been chosen so that it consists only of odd powers of $s$, so that $a\beta = \beta a^2$, while the rest of the terms commute with $a$. This makes the calculation easier to verify.

\begin{theorem} \label{counterexample}
    The element $\alpha$ is a unit in $\Z[G]$, and therefore defines a class $[\alpha] \in K_1(\Z[G])$. On this class, the following square fails to commute:
\[\begin{tikzcd}
	{K_1(\mathbb{Z}[G])} & {K_1(\mathbb{Z}[H])} \\
	{\{\pm1\} \oplus G^\mathrm{ab}} & {\{\pm1\} \oplus H^\mathrm{ab}}
	\arrow["{\Res}", from=1-1, to=1-2]
	\arrow["{\trc_G}"', from=1-1, to=2-1]
	\arrow["{\trc_H}", from=1-2, to=2-2]
	\arrow["p^!", from=2-1, to=2-2]
\end{tikzcd}\]
    Specifically, we have $p^! \trc_G[\alpha]  = (+1,0)$ and $\trc_H\Res[\alpha] = (+1,8)$ as elements of $\{\pm 1\} \oplus \Z/16$.
\end{theorem}

\begin{proof}
    Let us first remark that the inverse of $\alpha$ can be given explicitly as
    \[ \alpha^{-1} = 1+(1-s^8)(a-1)[3-5(s+s^{-1})-2(s^2+s^{-2})+12(s^3+s^{-3})+(14-10s^2-10s^{-2}) a^{-1}]. \]
    One can verify by direct calculation that $\alpha\alpha^{-1} = \alpha^{-1}\alpha = 1$. Later in the proof we will also give a second, more indirect method to check that $\alpha$ is invertible.
    
    We compute $\trc_G[\alpha]$ next. Note that $\epsilon(\alpha) = 1$, since $\epsilon$ is a ring homomorphism and $\epsilon(a-1) = 0$. Therefore the first coordinate of $\trc_G[\alpha]$ is $+1$, and the remaining coordinates are given by $\log(\alpha)$. To compute $\log(\alpha)$ we first pass from $\Z[G]$ to $\Z[G^\ab]$, which sends $a$ to $1$ and therefore kills $(a-1)$, hence $\alpha_{\mathrm{ab}} = 1$. Therefore $\trc_G[\alpha] = (+1,0) $. This is the identity element in $\{ \pm 1 \} \oplus \mathbb{Z}/16$ and hence applying $p^!$ gives the identity element $(+1,0) \in \{ \pm 1 \} \oplus \mathbb{Z}/16$. 

    We compute $\Res[\alpha]$ as the $K$-theory class of the $(3 \times 3)$ matrix $\Res(\alpha)$ given below. To be more specific, this is the matrix for multiplication \emph{on the left} by $\alpha$ on $\Z[G]$, using the basis for $\Z[G]$ as a right $\Z[H]$-module given by $\{1,a,a^2\}$.
\[
\operatorname{Res}(\alpha)=M=
\begin{pmatrix}
  s^8 - \beta & \beta & 1 - s^8 \\
  1 - s^8 + \beta & s^8 & -\beta \\
  0 & 1 - s^8 - \beta & s^8 + \beta
\end{pmatrix}.
\]
Notice the descending pattern of $s^8$ and $1-s^8$ in each column, and a corresponding ascending pattern of one copy of $-\beta$ and one copy of $\beta$ in each column; these move in opposite directions because of the relation $\beta a = a^{-1}\beta$.

The augmentation $\epsilon(\Res(\alpha))$ is the identity matrix, which we can easily calculate using the fact that $\epsilon(\beta) = 0$ and $\epsilon(s^8) = 1$. Therefore the first coordinate of $\trc_H\Res[\alpha]$ is $+1$, and the remaining coordinate is given by the log of the trace of $\Res(\alpha)$. The trace of $\Res(\alpha)$ is $3 s^8$ and hence the logarithm is $3 \cdot 8 = 24 \equiv 8 ~\text{mod}~16$. This finishes the calculation that $\trc_H\Res[\alpha] = (+1,8)$.

Finally, we can describe the second proof that $\alpha$ is invertible. It suffices to show that the matrix $\Res(\alpha)$ is invertible, since this matrix also describes the left multiplication action of $\alpha$ on one copy of $\Z[G]$. Since the coefficients for this matrix are in the commutative ring $\Z[C_{16}]$, we can take its determinant, which gives $\det(\Res(\alpha)) = 4 - 3s^8 - 3\beta^2$. We remark on the useful identity $s^8 \beta = - \beta$ for this computation.
We claim that $-20+21s^8+3\beta^2$ is the inverse in $\Z[C_{16}]$. Their product is $-143 + 144s^8 + 144\beta^2 - 9\beta^4$, but we can simplify this using the expressions
\[
\begin{array}{rcl}
    \beta^2 &=& 4 + 2s^2 - 2s^6 - 4s^8 - 2s^{10} + 2s^{14}, \\
    \beta^4 &=& 48 + 32s^2 - 32s^6 - 48s^8 - 32s^{10} + 32s^{14}.
\end{array}
\]
From this we deduce that $\beta^4 = -16 + 16s^8 + 16\beta^2$. This simplifies the product of $4 - 3s^8 - 3\beta^2$ and $-20+21s^8+3\beta^2$ to the expression
\[ -143 + 144s^8 + 144\beta^2 + 9(16 - 16s^8 - 16\beta^2) = 1. \]
Therefore $\Res(\alpha)$ is invertible, so $\alpha$ is invertible.
\end{proof}

\begin{remark}
    We have not described the provenance of the unit $\alpha$, and indeed this is one of the more difficult things to ascertain when conducting a computational search using an LLM. As far as we can tell, the unit has been constructed by the LLM via heuristics that suggest that a counterexample could come from a unit of the form $1 + (1 - s^8)(a - 1)(x + y(s + s^{-1}))$, then further tests are carried out to determine which integers $x$ and $y$ actually produce a counterexample. The heuristics used by the LLM are related to the methods used in \cite{jps_heuristic,bass_units_generalized}. The fact that the determinant of $\Res(\alpha)$ is invertible can also be interpreted using Bass cyclic units \cite{bass_units}, see also \cite{bass_units_new}. 

    Despite these connections, we are unable to find any direct source in the literature for these heuristics, making this an apparent example of genuine thinking and not mere copying on the part of the LLM. But we cannot be certain of this -- we are limited in our search by what the LLM chooses to share with us, and also by the limitations of a traditional web search.
\end{remark}

\section{An explicit counterexample to functoriality}

At this point we have proven the main theorem (Theorem \ref{main_thm}). But we can go further and distill out of $\alpha \in \Z[G]$ an explicit pair of perfect fibrations $p$ and $q$ such that $(q \circ p)^! \neq p^! \circ q^!$, by reverse-engineering Raptis's proof. (This step was also first carried out with the assistance of an LLM, but we like the following human version of the story better.)

Suppose $G$ is any finitely presented group, $H \leq G$ is any finite-index subgroup, and $\alpha \in \Z[G]^\times$ is any unit in the group ring. Fix any finite based CW complex $K$ such that $\pi_1(K) \cong G$, and let $V = K \vee S^2$. The element $\alpha = \textstyle\sum_i n_ig_i$ corresponds to a homotopy class of maps $\phi_\alpha\colon V \to V$ that is the identity along $K$, that pinches the extra copy of $S^2$ into finitely many copies of $S^2$, one for each non-zero term $n_ig_i$ in $\alpha$, and that sends each of these copies of $S^2$ back to $S^2$ by a map of degree $n_i$, together with the action of $g_i \in \pi_1(K)$ on it. This is well-defined up to homotopy, and the composition of maps of this form corresponds to multiplication in $\Z[G]$. Therefore, if $\alpha$ is a unit then $\phi_\alpha\colon V \to V$ is a homotopy equivalence. (However, crucially, $\phi_\alpha$ need not be a simple homotopy equivalence. Its Whitehead torsion, given by the class of $\alpha$ in $\mathrm{Wh}(G) = K_1(\mathbb{Z}[G])/ \{\pm G \}$, can be non-zero.) 

Let $Y$ be the mapping torus of $\phi_\alpha\colon V \to V$, and let $q\colon Y \to S^1$ be the projection to the circle coordinate. Since $\phi_\alpha$ acts trivially on $\pi_1(V)$, we have $\pi_1(Y) \cong G \times \Z$. Let $t$ denote the generator of $\pi_1(S^1)$, and also of this extra $\Z$ factor in $\pi_1(Y)$. Finally, let $X \to Y$ be the finite-sheeted covering space corresponding to the inclusion $H \times \Z \leq G \times \Z$.

\begin{theorem}
	When the perfect fibrations $p$ and $q$ are built in this way, we have
	\[ p^!q^![t] = (p^! \circ \trc_G[\alpha],[G:H] \cdot \chi(V)) \quad \textup{and} \quad (q \circ p)^![t] = (\trc_H \circ \Res[\alpha],[G:H] \cdot \chi(V)) \]
	as elements of $\pi_1(\Sigma^\infty_+ X) \cong (\{\pm 1\} \oplus H^\ab) \times \Z$. Here $\chi(V)$ denotes the Euler characteristic of $V$.
\end{theorem}

\begin{proof}
	Consider the following diagram, where the bottom-most map is the canonical inclusion of stable homotopy into $A$-theory.
	   \begin{equation}\label{two_transfers}
        \begin{tikzcd}
	 	& A(X) \rar{\trc_X} & \Sigma^\infty_+ X \\
	 	& A(Y) \rar{\trc_Y} \uar{\tr_p} & \Sigma^\infty_+ Y \uar{p^!} \\
		\Sigma^\infty_+ S^1 \rar & A(S^1) \uar{\tr_q}
	\end{tikzcd}
    \end{equation}
     The vertical maps in the middle are transfer maps on $A$-theory. Unlike the Becker--Gottlieb transfer, the transfer on $A$-theory is easily seen to be functorial, using the fact that pasting pullback squares produces a pullback square \cite[Section 3]{Raptis2021TransferMaps}. In other words, the composite $\tr_p \circ \tr_q$ is equal to $\tr_{q \circ p}$ in the homotopy category.
    However, the square on the right does not commute in general, and therefore the two routes through this diagram gives two different maps $\Sigma^\infty_+ S^1 \to \Sigma^\infty_+ X$.
    
    By a result of John Lind and the second author (\cite[Corollary 1.3]{lm}, see also \cite[Thm E]{CarmeliCnossenRamziYanovski2025}), the Becker--Gottlieb transfer $(q \circ p)^!$ is the composite along the upper route of \eqref{two_transfers}.
    Similarly $q^!$ is the composite from $\Sigma^\infty_+ S^1$ to $\Sigma^\infty_+ Y$, and therefore $p^!q^!$ is the composite along the lower route of \eqref{two_transfers}. Therefore, we only have to compute what each map in the diagram \eqref{two_transfers} does to the element
    \[ [t] = (+1,1) \in \pi_1( \Sigma^\infty_+ S^1) \cong \{\pm 1\} \oplus \Z. \]
    
    First this goes to the element of $\pi_1(A(S^1))$ represented by the self-equivalence of the one-point space over $S^1$ that moves the point around the circle using a zig-zag of the form $* \to I \leftarrow *$. Then we apply transfer to $\pi_1(A(Y))$. This gives the zig-zag $V \to V \times I \leftarrow V$ over the mapping torus $Y$, which applies the map $\phi_\alpha$ on the left and the identity of $V$ on the right. To get an element of $K_1(\Z[G \times \Z])$ out of this, we pull back to the universal cover $\widetilde Y$, which has a right $(G \times \Z)$-action, pass to cellular chains, then take the alternating sum of the resulting invertible $\Z[G \times \Z]$-matrices as elements of $K_1(\Z[G \times \Z])$.
    
    The pullback $V \times_Y {\widetilde Y}$ is a disjoint union of $\Z$ copies of $\widetilde V$, and the induced map on this space applies $\widetilde \phi_\alpha\colon \widetilde V \to \widetilde V$ to each copy and then increments the integer in $\Z$ by one. Taking chains to get a map of right $\Z[G \times \Z]$-modules, each cell in $K$ contributes a map that multiplies on the left by $t$, and therefore after accounting for signs these all contribute a $(1 \times 1)$ invertible matrix with entry $t^{\chi(K)}$. The remaining 2-cell is sent to itself with multiplication by $\alpha t \in \Z[G \times \Z]$. So all together, we get the element $[\alpha t^{\chi(V)}] \in K_1(\Z[G \times \Z])$.
	
	Now that we know where the element $[t]$ goes in $\pi_1(A(Y))$, we can apply the formulas from Lemmas \ref{lemma1}, \ref{lemma2}, and \ref{lemma3} to get its image under $(q \circ p)^!$ and $p^!q^!$. Since these are all homomorphisms, we can apply them to the terms $[\alpha]$ and $[t^{\chi(V)}] = \chi(V)[t]$ separately, and then add the results together at the end.
\end{proof}

\begin{corollary}
	For the choice of $G$, $H$, and $\alpha$ made in Theorem \ref{counterexample} above, we have $p^!q^![t] = (+1,0,3\chi(V))$ and $(qp)^![t] = (+1,8,3\chi(V))$, as elements of $\pi_1(\Sigma^\infty_+ X) \cong \{\pm 1\} \oplus \Z/16 \oplus \Z$.
\end{corollary}

Therefore we have constructed an explicit pair of perfect fibrations $p$ and $q$ such that $(q \circ p)^!$ and $p^! \circ q^!$ do not give the same homomorphism on $\pi_1$.

\begin{remark}
    We used $K$-theory to disprove the functoriality of the Becker--Gottlieb transfer. This might seem like an odd trick, but it is really more than just a trick. In retrospect, it was the correct way to attack the problem all along. What makes the counterexample run is that the monodromy of $Y \to S^1$ has nonzero Whitehead torsion $\alpha$, and is therefore not a simple homotopy equivalence. If it were a simple homotopy equivalence, then its Dwyer--Weiss--Williams index from \cite{dww} would vanish, allowing us to replace $Y \to S^1$ with a smooth fiber bundle with compact fiber. Then we would have $(q \circ p)^! = p^! \circ q^!$ by Dold's argument in \cite[Theorem 8.7]{dold_fixed_point_transfer}.
    
    The difference between simple and non-simple homotopy equivalences is measured by the Whitehead group $K_1(\Z[G])/(\{\pm 1\} \oplus G^\ab)$, so $K$-theory is exactly where we should look for algebraic data that can break the functoriality of the transfer.
\end{remark}

\bibliographystyle{amsalpha}
\bibliography{mybib}{}

\end{document}